\documentclass[11]{amsart}

\usepackage[pdftex,bookmarks,colorlinks,breaklinks]{hyperref}
\hypersetup{linkcolor=blue,citecolor=blue}
\usepackage{parskip}
\usepackage[T1]{fontenc}

\usepackage{tikz-cd} 

\usepackage{euscript}

\newtheorem{theorem}{Theorem}
\newtheorem*{theo}{Theorem}

\newtheorem{lemma}{Lemma}

\newtheorem{remark}{Remark}
\newtheorem{example}{Example}
\newtheorem{cor}{Corollary}

\newtheorem{prop}{Proposition}

\newcommand{\s}{\EuScript{X}}

\newcommand{\p}{\mbox{\bf P}}

\newcommand{\y}{\boldsymbol{y}}
\newcommand{\e}{\mbox{\bf E}}

\newcommand{\x}{\boldsymbol{x}}
\newcommand{\z}{\boldsymbol{z}}
\newcommand{\pa}{\partial}

\newcommand{\sz}{\s^{\mathbb{Z}_+}}
\newcommand{\Z}{{\mathbb Z}}
\newcommand{\N}{{\mathbb N}}
\newcommand{\R}{{\mathbb R}}

\newcommand{\T}{\mathcal{T}}
\newcommand{\I}{\mathcal{I}}
\newcommand{\pb}{\mathcal{PB}}
\newcommand{\tb}{\mathcal{TB}}
\newcommand{\cal}[1]{{\mathcal #1}}
\newcommand{\be}[1]{\begin{equation}\label{eq:#1}}
\newcommand{\ee}{\end{equation}}
\newcommand{\hpt}{\widehat{p^\tau}}
\newtheorem{assum}{Assumption}
\def\ol{\overline}

\begin{document}

\title{On Transformations of Markov Chains and Poisson Boundary}

\author{Iddo Ben-Ari}
\address{Department of Mathematics, University of Connecticut,  USA}
\email{iddo.ben-ari@uconn.edu}
\author{Behrang Forghani}
\address{Department of Mathematics, Bowdoin College, USA}
\email{bforghan@bowdoin.edu}
\maketitle

\begin{abstract}
A discrete-time Markov chain can be transformed into a new Markov chain by looking at its states along iterations of an almost surely finite stopping time. By the optional stopping theorem, any bounded harmonic function with respect to the transition function of the original chain is harmonic with respect to the transition function of the transformed chain. The reverse inclusion is in general not true. Our main result provides  a sufficient condition on the stopping time which guarantees that the  space of bounded harmonic functions for the transformed chain embeds in the space of bounded harmonic sequences for the original chain.  We also obtain a similar result on positive unbounded harmonic functions, under some additional conditions. Our work was motivated by and is analogous to \cite{Forghani-Kaimanovich2016}, the well-studied case when the Markov chain is a random walk on a discrete group.
\end{abstract}

\section{Introduction}
The classic Poisson formula naively says that a harmonic function on the unit disk in the complex plane, that is a function whose Laplacian vanishes, can be represented as an integral transform of its values on the  boundary of the disk. The integral transform is with respect to a kernel known as the Poisson kernel.  Probabilistically, the Poisson kernel is the distribution of the position where Brownian motion exits the open disk. The idea of representing a harmonic function as an integral transform of its boundary values extends beyond Laplacian to other contexts. In particular, in the theory of Markov chains, the study of notion of Poisson formula  goes back to the works of  Blackwell \cite{Blackwell1955} and Feller \cite{Feller1956}. 

Let $\s$ be an infinite, countable set. This will serve as our state space. Let $p:\s \times \s \to [0,1]$  be a transition function, that is $\sum_y p(x,y)=1$ for all $x\in \s$.   A function $f:\s \to \R$ is $p$-harmonic if $\sum_y p(x,y) |f(y)|<\infty$ for all $x\in \s$ and it also satisfies the  mean value property
$f(x)=\sum_y p(x,y)f(y)$ for all $x\in\s$.  Note that the constant functions are bounded and  $p$-harmonic, and that the set of $p$-harmonic functions is a linear space. It is easy to see that the space of bounded $p$-harmonic functions is a Banach space with respect to the  sup-norm.  By Rohlin's theory of measurable partitions  \cite{Rohlin52} and Doob's theory of martingales \cite{Doob53}, there exists a probability space called the Poisson boundary such that its $L^\infty$ as a Banach space is isometrically isomorphic to the space of bounded $p$-harmonic functions, see \cite{Kaimanovich1996}.

There are  extensive developments of the theory of Poisson boundaries whenever the state space $\s$ has some special structure, e.g. a group, a Riemannian manifold,  see \cite{Furstenberg1963}, \cite{Kaimanovich2002}, \cite{Kaimanovich-Woess2002} and the references therein. 

Furstenberg \cite{Furstenberg1973} showed that the Poisson boundary of a random walk on a group is isomorphic to the induced random walk to a recurrent subgroup. Later, Kaimanovich \cite{Kaimanovich83}, Muchnik \cite{Muchnik2006}, Willis \cite{Willis1990} provided more constructions on a random walk on group that preserve the Poisson boundary. The most general method up to date to construct random walks on groups with a common  Poisson boundary was recently introduced  by  Kaimanovich and the second author in \cite{Forghani-Kaimanovich2016}, \cite{Behrang2016}.  Their method is based on applying a randomized stopping time to the space of sample paths to obtain a new random walk with identical  Poisson boundary.  A crucial step in the proof is that each countable group can be viewed as a quotient space of some free semigroup. Hence the proofs are  applicable only  to random walks on countable groups. This approach was employed to study the space of positive harmonic functions on a countable group with respect to a stopping time \cite{Forghani-Mallahi2016}.

In this paper, we consider the case when the state space has \emph{no additional structure}.  Given a Markov chain $\x = (x_0,x_1,\dots)$ on a countable state space  $\s$ with transition function $p$ and a stopping time $\tau$ which is almost surely finite (see equation \eqref{eq:finiteness}), 
we consider the process obtained by looking at the Markov chain corresponding to $p$ along iterations of the stopping time, a sequence we denote by $\langle \tau (\x) \rangle$.   Through the strong Markov property, see \cite{Revuz1984}, this  process is a Markov chain on $\s$ which we call the transformed chain, and whose transition function we denote by $p^\tau$, given by $p^\tau (x,y) = \p_x (x_\tau = y)$.  We say  $\tau$ can be asymptotically recovered (Assumption~\ref{as:asymp}) when there exists a positive integer-valued map $\rho$ on the space of sample paths such that  
$$
\displaystyle\limsup_{n\to\infty}  \inf_{x} \p_x \bigg(n+\rho (x_n,x_{n+1},\cdots) \in \langle \tau(\x) \rangle\bigg)=1.
$$

A simple example for a stopping time that can be asymptotically recovered is that of a hitting time. More examples are given in  Section \ref{sec:examples}.

We investigate how the Poisson boundary (bounded harmonic functions) of the original and transformed Markov chains are related. We do this by showing the intuitively clear fact that  if the stopping time can be asymptotically recovered  (Assumption~\ref{as:asymp}), then the space of bounded harmonic functions for the transformed process is  embedded in the space of space-time harmonic functions for the original chain.  This is the statement  of our main result,  Theorem \ref{th:newmain}:
\begin{theo}
 Let  $p$ be  a transient  transition function on $\s$, and suppose that $o\in \s$ is such that all $x\in \s-\{o\}$ are accessible from $o$.  Let $\tau$ be a stopping time for which is finite a.s. under any initial distribution and can  be   asymptotically recovered (Assumption~\ref{as:asymp}).
Then for any positive bounded $p^\tau$--harmonic function $u$ there exists an extension $\bar{u}$ to   $\s\times \Z_+$,  ${\bar u}(x,0)=u(x)$, such that
  \begin{enumerate} 
 \item  $\bar{u}$ is a positive bounded $p$-harmonic sequence. 
 \item $\|{\bar u} \|_\infty =\|u\|_\infty$. 
 \end{enumerate} 
 \end{theo} 
 We note that in general, and unlike the case of random walks on groups, the Poisson boundaries of the original and transformed  chains may be fundamentally  different, see the example in Section \ref{splitting}.   In Theorem \ref{th:martin} we extend the scope to positive harmonic functions under some additional conditions. 

Our proof of the theorem is based on the construction of the \emph{Martin boundary}, which is one the main qualitative space in boundary theory and potential theory associated to Markov chains. The Martin boundary of a Markov chain is the  topological counterpart of the Poisson boundary which is responsible for representation of positive harmonic functions. If the Martin boundary as a Borel space equipped with an appropriate probability measure, then it is isomorphic (as a measure space) with the Poisson boundary, see \cite{Dynkin69}. 

The organization of the paper is as follows.  In Section \ref{sec:preliminaries} we recall the theory of  Poisson, tail, and Martin boundaries. Section 3 devoted to constructing Markov chains via transformation.  In  Section 4 we show how the tools from the theory of Martin boundary can be applied to the transformed Markov chains via stopping times. 
Our main result is  proved in Section \ref{sec:main}, and in Section \ref{sec:examples} we present a number of examples. Two standard approximation results used in the proof of Theorem \ref{th:martin} are  proved in the Appendix.

\subsection*{Acknowledgements}
We would like to thank the anonymous referee for reading  the manual very carefully and suggesting improvements.

\section{Preliminaries}\label{sec:preliminaries} 
\subsection{Markov chains}
\label{sec:MC_defn} 
Let $\s$ be a countable set. The set $\s$ and its power set form a measurable space.  Let $\s^{\Z_+}=\{\x=(x_0,x_1,\dots):x_n \in \s,n\in\Z_+\}$, the set of $\s$-valued sequences indexed by $\Z_+$. For every $n\in\Z_+$ define the coordinate function $\omega_n(\x)=x_n$. Denote by $\cal{F}_k^\infty(\s)=\sigma(\omega_k,\omega_{k+1},\cdots)$ the sigma-algebra generated by the coordinate functions $\omega_i,~i\ge k$. Let ${\cal F^\infty}=\cal F^\infty_0(\s)$.
The measurable  space  $(\sz,{\cal F^\infty})$ is called the  space of sample paths. For our work, we will also need to define an auxiliary process  $\z=(z_n:n\in \Z_+)$, as follows. Given $\x$ and $t \in\Z_+$, we let $ z_n = (x_n,t+n)$. That is, $\z$ also keeps track of time. 

Let $p$ be a transition function on $\s$. That is $p:\s \times \s \to [0,1]$ and $\sum_y p(x,y) =1$ for all $x\in \s$. Let $m$ be a probability measure on $\s$.  For $n\in \Z_+$, define the $n$-th iteration of $p$, denoted by $p^n$, through 
\begin{equation}
\label{eq:p_iterated} 
p^0(x,y) = \delta_{x}(y), \mbox{ and }p^{n+1} (x,y) = \sum_z p(x,z) p^n (z,y)
\end{equation}
(note that $p^1=p$, and that $p^n$ is also a transition function). By Kolmogorv's extension theorem, there exists a unique probability measure  $\p_m$ on the space of sample paths satisfying
$$
\p_m(a_0,a_1,\cdots,a_n)=m(a_0)p(a_0,a_1)\cdots p(a_{n-1},a_n)
$$ 
where $(a_0,a_1,\cdots,a_n)=\{\x\in\sz\ :\ \omega_i(\x)=a_i,\ i=0,\cdots,n\}$. The probability measure 
$m$ is usually referred to as the initial distribution under $\p_m$. As usual, we write $\e_m$ for the expectation operator associated with $\p_m$, also for $x\in \s$ we abbreviate and write $\p_x,\e_x$ instead of $\p_{\delta_x},\e_{\delta_x}$. Note then that 
$$
\p_m=\sum_x m(x)\p_x,\ \ \ \  \e_m=\sum_xm(x)\e_x.
$$
The triple $(\s,p,m)$ is called a Markov chain on the state space $\s$ with the transition function $p$ and the initial distribution $m$.
\subsection{Harmonic functions}
Suppose that $f:\s \to \R$ satisfies $\sum_y p(x,y) |f(y)|< \infty$ for all $x\in \s$. Then we can define a function $pf:\s\to \R$ through 
$$pf(x)=\sum_yp(x,y)f(y).$$ 
If $pf =f$, then $f$ is called $p$--harmonic.  We denote the set of all bounded $p$--harmonic functions and the set of all positive $p$--harmonic functions by $H^{\infty}(\s,p)$ and $H_{+}(\s,p)$, respectively.  We also write $H^\infty_+(\s,p)$ for the convex cone of bounded positive harmonic functions.  
\subsection{Harmonic sequences}\label{Z}
A sequence of functions $(f_n:n\in \Z_+)$, where  $f_n: \s\to\R$, is called a $p$--harmonic sequence whenever 
$$pf_{n+1}=f_n.$$
The space $p$--harmonic sequence is denoted by $S(\s,p)$, while the subspace of  bounded $p$--harmonic sequences, and nonnegative $p$--harmonic  sequences are denoted by  $S^{\infty}(\s,p)$ and  $S_+(\s,p)$, respectively.  If $f$ is a nonnegative $p$--harmonic, then  $(f,f,\dots)$ is a nonnegative $p$--harmonic sequence.  Harmonic sequences sometimes are called space--time harmonic functions. Indeed, given a $p$--harmonic sequence $(f_n:n\in\Z_+)$, define a function $f:\s\times \Z_+\to \R$ by letting $f(x,n)=f_n(x)$.  Now if one defines a transition function  $p^+$ on $\s\times \Z_+$ by letting
$$p^+((x,n),(y,n+1))=p(x,y),$$
then $p^+ f = f$. Conversely, if $f$ is $p^+$--harmonic, then $(f(x,0),f(x,1),\dots)$ is a $p$-harmonic sequence. In terms of notation, ${\cal S}(\s,p)=H(\s\times \Z_+,p^+)$. For $(x,t)\in \s\times\Z_+$, write  $\p_{x,t}$ for the probability measure on the space of sample paths on $\s\times \Z_+$ induced by  $p^+$ with initial distribution $\delta_{x,t}$.  A sample path in that space will be written as  $\boldsymbol z=(z_0,z_1,\cdots)$.
\subsection{Poisson boundary}\label{sec:Poisson}
Let $S:\sz\to\sz$ be the time-shift, that is 
$$
S(x_0,x_1,\cdots)=(x_1,x_2,\cdots),
$$
and for $k\ge 1$ write $S^k$ for the $k$-th iteration of $S$, i.e, $S^k (x_0,x_2,\cdots) =(x_k,x_{k+1},\cdots)$.
The sigma-algebra  $\I=\{A\in\mathcal F^\infty\ :\ S^{-1}A=A\}$ is called the  invariant sigma-algebra. Let $\theta$ be a probability measure on $\s$ with full support, that is $\theta(x)>0$ for all $x\in \s$. Let $\overline{\I}(\s,p)$ denote the completion of ${\cal I}$ with respect to the probability measure $\p_{\theta}$. As can be easily seen, this completion is  equivalent to requiring that whenever $A\subseteq \sz$ is such that $A\subseteq B$ for some $ B \in{\cal F^\infty}$, satisfying $\p_x (B) = 0$ for all $x\in \s$, then $A$ is measurable with respect to the completion. Therefore $\overline{\cal I}(\s,p)$ is independent of the particular choice of $\theta$.   The Poisson boundary is defined as the  restriction of $(\sz,\cal F^\infty,\p_{\theta})$ to  $\overline{\I}(\s,p)$.  Denote the Poisson boundary with respect to the transition function $p$ as $\mathcal{PB}(\s,p):=(\s^{\Z_+}, \overline{\I}(\s,p), \p_\theta)$.  As an ergodic theoretic object, the Poisson boundary is identified with  the space of ergodic components of the time-shift on the space of sample paths \cite{Kaimanovich91}.  

The Poisson boundary identifies the space of bounded harmonic functions. More precisely, let $f$ be a bounded  $p$--harmonic function, define 

\begin{equation}\label{eq:martin-approach}
\begin{array}{rcl}
H^\infty(\s, p) &\longrightarrow &L^\infty(\mathcal{PB}(\s,p))\\
f  &\longmapsto&  \phi_f(\x):= \displaystyle\lim_{n\to\infty} f(x_n)  \ \mbox{ a.e.} 
\end{array}
\end{equation}

Because $f$ is a bounded $p$--harmonic function, the sequence $(f(x_n))_{n}$ is $\p_\theta$--martingale, therefore $\phi_f(x)$ almost surely exists. Because $\overline{\I}(\s,p)$ is a complete sigma--algebra, $\phi_f$ is measurable with respect to the probability space $\pb(\s,p)$.  We can define the inverse as follows.
\begin{equation}\label{eq: bounded-infinity}
\begin{array}{rcl}
L^\infty(\pb(\s,p)) &\longrightarrow &H^\infty(\s, p) \\
\phi  &\longmapsto&  f_\phi(x):= \int \phi(\x) \ d\p_x(\x)
\end{array}
\end{equation}
Moreover the definitions of $\phi_f$ and $f_\phi$ imply that $\|f\|_\infty=\|\phi_f\|_\infty$, hence $H^\infty(\s,p)$
is isometrically isomorphic to $L^\infty(\pb(\s,p))$, see also Corollary~\ref{cor:isometry}.

\subsection{Tail boundary}
Consider the  tail sigma--algebra $\T=\displaystyle\bigcap_{k=0}^{\infty}S^{-k}(\cal F^\infty)$  on the space of sample paths. The completion of $\T$ with respect to $\p_\theta$, where $\theta$ is as in the last paragraph, is denoted by $\ol\T(\s,p)$. The tail boundary is the  restriction of $(\sz,\cal F,\p_{\theta})$  to the sigma-algebra  $\ol\T(\s,p)$. The tail boundary associated with the transition function $p$ is denoted by $\tb(\s,p):=(\s^{\Z_+}, \overline{\T}(\s,p), \p_\theta))$. Similarly to the Poisson boundary, the space of bounded  $p$--harmonic sequences $S^\infty(\s,p)$ is isometrically isomorphic to $L^\infty(\tb(\s,p))$:
$$
\begin{array}{rcl}
S^\infty(\s, p) &\longrightarrow &L^\infty(\tb(\s,p))\\
F=(f_n)_n  &\longmapsto&  \psi_F(\x):= \displaystyle\lim_{n\to\infty} f_n(x_n)\ a.e
\end{array}
$$
Using the fact that the space of bounded $p^+$--harmonic functions can be viewed as the space of bounded $p$--harmonic functions implies $H^\infty(\s\times \Z_+,p^+)$ is isometrically isomorphic to $L^\infty(\tb(\s,p))$. On the other hand, $H^\infty(\s\times \Z_+,p^+)$ is isometrically isomorphic to 
$L^\infty(\pb(\s\times\Z_+,p^+))$. Summarizing: we have $L^\infty(\tb(\s,p))$  is isometrically isomorphic to $L^\infty(\pb(\s\times\Z_+,p^+))$. Therefore,  the tail boundary associated with $p$ is isomorphic (as a probability space) to the Poisson boundary associated with $p^+$, for more details see \cite{Ka92,Kaimanovich1996}

$$ 
 \begin{tikzcd}
H^\infty(\s\times\Z_+,p^+)\arrow[r, "\cong"]  & S^\infty(\s,p)  \dar{\cong} \\  
L^\infty(\pb(\s\times\Z_+,p^+)) \arrow[u, "\cong"]   & L^\infty(\tb(\s,p))
\end{tikzcd}
$$

 Viewing the Poisson boundary as representing the  bounded harmonic functions,  the tail boundary as representing the  bounded harmonic sequences, and using the fact that  every harmonic function uniquely extends to a harmonic sequence, we can think of the Poisson boundary as a subset on the tail boundary. We will not get into any details here.   However, we have the following:

$$ 
 \begin{tikzcd}
H^\infty(\s,p)\arrow[r, hook]  & S^\infty(\s,p)  \dar{\cong} \\  
L^\infty(\pb(\s,p)) \arrow[u, "\cong"]   & L^\infty(\tb(\s,p))
\end{tikzcd}
$$
 
 A transition function of a Markov chain is called steady whenever the tail boundary  coincides mod $\p_x$ with  the Poisson boundary  for any $x$ in $\s$, and, in particular, all bounded harmonic sequences are bounded harmonic functions.  The ``0-2'' law  determines whether a transition function is steady, see  \cite{Ka92} for more details. Here is one sufficient condition: 
 
\begin{example}[\cite{Ka92}]
For $x\in \s$, let  $g_x = \inf \{n\ge 1: p^n(x,x)>0\}$. Then $p$ is steady if the greatest common divisor of $\{g_x>0:x\in \s\}$ is $1$.  
\end{example}

\subsection{Martin boundary}
\label{sec:martin_boundary}
The relation between the Poisson and the tail boundaries to the invariant and tail sigma-algebras allow  to characterize the space of bounded harmonic functions and bounded harmonic sequences, respectively.  We now introduce the Martin boundary, a topological boundary also used to characterize positive harmonic functions (Theorem \ref{th:doob59} below), and which is more suitable for our purposes.  We comment that the Poisson boundary can be identified as a subset of the Martin topology equipped with an appropriate probability measure \cite{Kaimanovich1996}, \cite{Sawyer1997}, and \cite{Woess09}, also see Corollary \ref{cor:isometry} below.  One can refer to \cite{Derriennic1976}, \cite{Kaimanovich1996}, \cite{Sawyer1997} and  \cite{Woess09} for the construction of Martin boundary.  In this section, we remind the reader about the definition of the Martin boundary and related results which will be used  later.

Let $p$ be a transition function on $\s$.  The  transition function is called transient whenever the Green's function $G^p(x,y)=\sum_{n\geq0}p^n(x,y)$ is finite for all $x,y\in \s$, $p^n$ is the $n$-th iteration of $p$, defined in \eqref{eq:p_iterated}. 

We always make the following assumption on $p$: 
\begin{assum}~
\label{as:required} 
\begin{enumerate}
\item $p$ is transient. 
\item There exists a state $o$ such that all $y \in \s-\{o\}$ are accessible from $o$:  
$$G^p(o,y)>0\mbox{ for all }y\in \s.$$  
\end{enumerate}
\end{assum} 

We also define  the {\it Martin kernel} on $\s\times \s$
$$K^p(x,y)=\begin{cases} \frac{G^p(x,y)}{G^p(o,y)} & G^p(o,y)>0 \\ 0 & \mbox{otherwise}\end{cases}$$ 
 The  {\it Martin compactification} of $\s$ is the topological space $M(\s,p)$

 satisfying the following requirements: 
 \begin{enumerate}
 \item Every singleton $\{x\},~x\in \s$, is open. 
 \item $\s$ is dense. 
\item For  $x\in \s$, the function $K^p(x,\cdot)$ extends to a continuous function on $M(\s,p) $, and the  set of extensions separate points in $M(\s,p) \backslash \s$. 
\end{enumerate}
These  requirements uniquely determine a compact topological space (up to homeomorphism). Furthermore, the resulting space is metrizable \cite{Woess09}. The compact topological space $\partial(\s,p)=M(\s,p) \backslash \s$ is called the \emph{Martin boundary} of the Markov chain with respect to the transition function $p$. 

The \emph{minimal Martin boundary} is the Borel subset $\partial_m(\s,p)$ of $\partial(\s,p)$  consisting of all $\xi$ satisfying 
\begin{enumerate} 
\item $K^p(\cdot, \xi) \in H_+(\s,p)$. 
\item $K^p(\cdot,\xi)$ is {\it minimal harmonic}: if $u\in H_+(\s,p)$ and  $u \le K^p(\cdot,\xi)$, then  $u = c K^p(\cdot, \xi)$ for some $c\le 1$. 
\end{enumerate} 
\begin{theorem}\cite{Doob59}\label{th:doob59}
Let $u\in H_+(\s,p)$. Then  there exists a unique finite measure $\mu_u$ on the Borel sigma-algebra on $\partial_m(\s,p)$ such that 
$$
u(x)=\int_{\pa_m(\s,p)}K^p(x,\xi)d\mu_u(\xi).
$$
\end{theorem}
The measure $\mu_u$ is called the representation of $u$. Since    $K^p(o,\xi)=1$ for all $\xi\in \partial_m (\s,p)$, we have that  $u(o) = \mu_u (\partial_m(\s,p))$. Note that we can consider $\mu_u$ as a finite measure on the compact metric space $\partial (\s,p)$, by letting $\mu_u(\partial (\s,p)-\partial_m(\s,p))=0$. 
 A special role is reserved for $\mu_{\bf 1}$ the representation of the constant function ${\bf 1}$. This is due to the following two results: 
\begin{cor}\label{cor:isometry}
The mapping  $T$ given by 
\begin{equation}\label{eq:mapping}  (T f)(x) = \int_{\partial_m(\s,p)} K^p(x,\xi) f (\xi) d\mu_{\bf 1}.
\end{equation}  
defines a linear isometry from $L^\infty(\mu_1)$ onto $H^\infty (\s,p)$ and also $L^\infty(\mathcal{PB}(\s,p))$.
\end{cor} 
\begin{proof}
The right-hand side of \eqref{eq:mapping} defines a  linear mapping from $L^\infty(\mu_{\bf 1})$ to the linear space of bounded real-valued functions on $\s$ equipped with the $\sup$-norm. Also,  
 $$|Tf(x) | \le \|f\|_\infty \int_{\partial_m(\s,p)} K^p(x,\xi) d \mu_{\bf 1}(\xi) = \|f\|_\infty.$$ 
 Therefore $\|Tf \|_\infty \le \|f\|_\infty$. 
  By dominated convergence, 
$$ p (T f)(x) =  \int_{\partial_m(\s,p)}\sum_y p(x,y)  K^p(y,\xi) f (\xi) d \mu_{\bf 1}(\xi)=  \int_{\partial_m(\s,p)} K^p(x,\xi) f(\xi) d \mu_{\bf 1} (\xi)=(Tf)(x),$$ 
therefore $Tf \in H^\infty (\s,p)$. Next we show that $T$ is an isometry. Suppose first that $f\in L^\infty(\mu_{\bf 1})$ is nonnegative, and let $u = Tf$. Then since $\|u\|_{\infty}{\bf 1}-u$ and $u$ are both in $H_+(\s,p)$, it follows from the uniqueness assertion in Theorem \ref{th:doob59}, that $\|u\|_\infty \mu_{\bf 1} =\mu_{ \|u\|_{\infty}{\bf 1}-u}+ \mu_u$, the sum of two positive measures. Therefore, not only $\mu_u \ll \mu_{\bf 1}$, but also, $0\le  \frac{ d \mu_u}{d \mu_{\bf 1}} \le \|u\|_\infty$. Since $\frac{d \mu_u}{d\mu_{\bf 1} } =f$, we have $\|f\|_\infty \le \|u\|_\infty$. In the case of signed $f$, this means 
$$
\| \left .  \|f\|_\infty{\bf 1} \pm f\right.  \|_\infty \le \| \left. \|f\|_\infty {\bf 1} \pm  Tf \right.\|_\infty.
$$ 
The righthand side is bounded above by $\|f\|_\infty + \|Tf \|_\infty$. As for the left hand side, we can choose the sign so that the norm is equal to $2\|f\|_\infty$. Therefore, $\|Tf \|_\infty \ge \|f\|_\infty$, and since the reverse inequality is already established,  $T$ is an isometry. Finally, we show that $T$ is onto. If  $v\in  H^\infty_+(\s,p)$, then as already seen,  $\mu_v \ll \mu_{\bf 1}$, and $\frac{d\mu_v}{d\mu_{\bf 1}} \in L^\infty(\mu_{\bf 1})$, so that $v$ is in the range of $T$.  If $u \in H^\infty(\s,p)$, then we can write it as a difference of two elements in $H^\infty_+(\s,p)$, i.e, $u=(\|u\|_\infty{\bf 1}  + u) - \|u\|_{\infty}$.  Therefore, $u$ is also in the range of $T$. 
\end{proof} 
\begin{theorem}\cite{Dynkin69,Woess09}
\label{th:limits}
\begin{enumerate} 
\item There exists a $\partial(\s,p)$-valued random variable $x_\infty$ such that for $\p_x$-a.s. sample path $\lim_{n\to\infty} x_n=x_\infty $ is in the Martin topology for all $x$ in $\s$. 
\item  The random variable $x_\infty$ is supported on $\partial_m(\s,p)$ and for every measurable set $A$ in $\partial_m(\s,p)$,   
$$
\p_x (x_\infty \in A) = \int_A K^p(x,\xi) d \mu_{\bf 1} (\xi).
$$
\end{enumerate} 
\end{theorem}

\section{Transformed Markov chains}\label{sec:transformed}
\subsection{Stopping time}
\label{sec:stopped_process}
A measurable function $\tau:\s^{\Z_+}\to \Z_+\cup\{\infty\}$ is called stopping time, if for every $k\in \Z_+$, the set $\{ \x: \tau(\x)=k\}$ is a measurable set in the sigma-algebra generated by the first $k+1$ coordinate function $\sigma(\omega_0,\omega_1,\cdots\omega_k)$. In what follows, we will assume that 
\begin{equation} 
\label{eq:finiteness} 
\p_x(\tau < \infty)=1 \mbox{  for all }x\in \s.
\end{equation}
Given a stopping time $\tau$ satisfying \eqref{eq:finiteness}, a nondecreasing sequence is induced by iteration:  
$$
\tau_0 =0,\ \ ~\tau_1=\tau,\ \  \tau_{n+1} = \begin{cases} \tau_n + \tau \circ  S^{\tau_n} & \tau_n <\infty; \\ \infty & \mbox{otherwise.} \end{cases}$$ 
With this sequence, we obtain a transformed process $\x^{\tau}=(y_n(\x):n\in\Z_+)$ given by $y_n = x_{\tau_n}$. 
By the strong Markov property, see \cite{Revuz1984}, 
$$ \p( x_{\tau_{n+1}}=z | \sigma(x_{\tau_1}\cdots,x_{\tau_n} )) = \p_{y_n} (x_{\tau} = z).$$
Therefore $\x^{\tau}$ is a Markov chain with  the transition function 
$$
p^{\tau}(x,y)=\p_x (x_{\tau} = y).
$$

Note that Doob's optional stopping theorem implies that for any stopping time $\tau$, we can write
$$
H^{\infty}(\s,p)\subseteq H^{\infty}(\s,p^\tau).$$
Similarly,   
$$H^{\infty}(\s\times\Z_+,p^+)\subseteq H^{\infty}(\s\times\Z_+,(p^{+})^\tau).
$$

Let $\p^\tau$ denote the probability measure on the space of sample paths with respect to the transition function $p^\tau$. We could map almost every sample path with respect to $p$ to  a sample path with respect to $p^\tau$:
$$
\begin{array}{rcl}
(\s^{\Z_+},\p_\theta)&\longrightarrow &(\s^{\Z_+},\p^\tau_\theta)\\
\x=(x_n)_n &\longmapsto & \x_\tau:=(x_{\tau_n})_n,
\end{array}
$$
which implies $L^\infty(\pb(\s,p))$ is isomorphic to a subspace of $L^\infty(\pb(\s,p^\tau))$.

Consider the following:   
\begin{assum} \label{as:asymp} 
There exists a mapping $\rho:{\s}^{\Z_+}\to \Z_+$ such that
$$
\displaystyle\limsup_{n\to\infty}  \inf_{x} \p_x (\rho_n (\x) \in \langle \tau(\x) \rangle)=1,
$$
where $\rho_n({\x}) = n + \rho \circ S^n({\x})$ and  $\langle \tau(\x) \rangle=(\tau_0(\x),\tau_1(\x),\cdots)$. 
\end{assum} 
In order to be able to employ the tools from the last section, we need to insure that $p^\tau$ satisfies the conditions of Assumption~\ref{as:required}. 
We observe that if $\x$ (equivalently, $p$) is transient then so is $\x^\tau$ (equivalently $p^\tau$).

\begin{lemma}
If $p$ is transient, then so is  $p^\tau$.  Furthermore, for all $x$ and $y\in \s$, we have $G^{p^\tau}(x,y) \le G^p(x,y)$. 
\end{lemma}
\begin{proof}
	Transience of $p^\tau$ is equivalent to $\x^\tau$ visiting each state finitely often under $\p_x$ for all $x\in \s$. Since this holds for $\x$, and the paths of $\x^\tau$ are subsequences of $\x$ both statements hold.  
\end{proof}	

Note that $p^\tau$ may, in general, not satisfy the second condition in Assumption \ref{as:required}. For example let $\s=\Z_+$, the set of nonnegative integers and let  $p(n,n+1)=1$, then $G^p(0,n)=1>0$ for any natural number $n$. If we consider the stopping time $\tau=2$, then for any $n\in\Z_+$, $G^{p^{\tau}}(m,n)>0$ if and only if $n-m \in 2\Z_+$. Therefore there does not exist $m\in\Z_+$ such that $G^{p^{\tau}}(m,n)>0$ for all $n\in\Z_+$.  In Section~\ref{sec : extension}, we will remedy this by expanding the state space.  

\section{The extension} \label{sec :  extension}
Assumption \ref{as:asymp} does not warrant that  $p^\tau$ satisfies the second condition of Assumption \ref{as:required} which is required for defining the Martin boundary. If it does, we need not do anything. Otherwise, we need to introduce the following completion. 

Our starting point is a transition function $p$ on $\s$ satisfying Assumption~\ref{as:required},  and a stopping time $\tau$ for $p$ satisfying Assumption~\ref{as:asymp}.

The first step is to append a state to $\s$, and extend $p$ to the new resulting extended state space. Let $\s^* = \s \cup\{*\}$, where $*$ is a state not in $\s$. Let $\theta$ be any probability measure on $\s^*$ with $\theta(x)>0$ for all $x\in \s$ and $\theta(*)=0$.  Extend $p$ to $\s^*$ by letting
$$  p^*(x,y) = \begin{cases} p(x,y)& ~x,y\in \s\\ 
\theta(y)& x=* \\ 
0 & \mbox{otherwise}. 
\end{cases}$$
We write  $\x^*$ for the corresponding Makrov chain. 
Clearly $p^*$ satisfies both conditions in Assumption~\ref{as:required} with $\s$ replaced by $\s^*$ and $o=*$. We also write $\p^*_x $ and $\e^*_x$ for the distribution and corresponding expectation associated with $\x^*$, starting from $x$ in $\s^*$. Note that for any $x\in \s$, we have $\p^*_x$ is supported on $\s$-valued sequences and coincides with $\p_x$. It could be therefore viewed as an extension of $\p_x$. 

This extension preserves the space of bounded harmonic functions: 
\begin{lemma}
Let $f$ be a bounded  function on $\s$. Then, $f$ is $p$--harmonic if and only there exists a unique bounded function $f^*$ on $\s^*$ such that 
\begin{enumerate} 
\item $f^*$ is an extension of $f$, that is $f^*(x)=f(x)$ for all $x$ in $\s$,
\item $f^*$ is $p^*$--harmonic. 
\end{enumerate} 
\end{lemma}
\begin{proof}
It is enough to define $f^*(*) = \sum_{y\in \s} \theta(y) f(y)$ and $f^*(x)=f(x)$ for any $x$ in $\s$. Them $f^*$ is an extension of $f$ and a bounded $p^*$--harmonic. The reverse direction is clear.
\end{proof}
Next we extend the stopping time $\tau$ to   $\s^*$-valued sequences by setting 
 \[ \tau^*(\x^*) = \begin{cases} \tau(\x^*)&\mbox{if } x_0^*,x_1^*,\dots \in \s \\
  1+ \tau(x_1^*,x_2^*,\dots)&\mbox{if }  x_0^*=*,x_1^*,x_2^*,\dots \in \s\\
  \min\{j:x_j^*= *\}&  \mbox{ otherwise} 
 \end{cases} 
 \]

It immediately follows that  $\tau^*$ is a stopping time for $\x^*$. Note that under $\p_x^*$ for  $x\in \s$, we have $\tau^*=\tau$ a.s. Furthermore,   if $x\in \s^*$,  then $\tau^*$ is finite a.s. $\p^*_x$. As a result, and similarly to the definition of $p^\tau$, we have an induced transition function $(p^*)^{\tau^*}$, defined as follows: 
$$ 
(p^*)^{\tau^*} (x,y) = \p_x^* (\x^*_{\tau^*}=y) = \begin{cases} p^{\tau}(x,y) & x,y \in \s \\ \sum_{x'}\theta(x') p^\tau(x',y) & x =*,y \in \s\\
0 & \mbox{otherwise} \end{cases}
$$

We will write $\y^*$ for the induced chain, that is $y^*_n = \x^*_{\tau^*_n},~n\in\Z_+,$ 
where, 
$$
\tau^*_0=0\ \mbox{ and }\ \tau^*_{n+1} = \tau^*_n+\tau^*\circ S^{\tau^*_n}$$ 
Finally, we restrict $(p^*)^{\tau^*}$ to $\widehat \s$, the subset of  all states which can be reached from any state under $(p^*)^{\tau^*}$ and the state $*$. More precisely, $\widehat \s$, defined as follows: 
$$\widehat\s  = \{y \in \s^*: (p^*)^{\tau^*}(x,y) >0\mbox{ for some }x\in \s^* \}\cup\{*\}.$$
Equivalently, 
\begin{align*} \widehat \s 
 & = \{y\in \s : p^{\tau}(x,y)>0\mbox{ for some }x \in \s\}\cup \{*\}.
 \end{align*} 
 
 Denote the restriction of $(p^*)^{\tau^*}$ to $\widehat \s$ by $\hpt$.  The sample paths with respect to the transition function $\hpt$ will be denoted by $\widehat\y=(\widehat y_0,\widehat y_1,\widehat y_3,\cdots)$. By construction, the Markov chain associated with the transition function $\hpt$ on $\widehat \s$ satisfies Assumption~\ref{as:required}.

 \section{Main Result}\label{sec:main}
 We are ready to state our main results. 
 \begin{theorem}
 \label{th:newmain}
  Suppose that $p$ is a transition function on $\s$ satisfying Assumption \ref{as:required} and that $\tau$ is a stopping time on $\s$-valued sequences satisfying Assumption \ref{as:asymp}. 
  Then for any $0\le u \in H^\infty (\s,p^\tau)$, there exists a function  $\bar{u}$ on   $\s\times \Z_+$ such that
  \begin{enumerate} 
  	\item ${\bar u}(x,0)=u(x),~x\in \s$. 
  \item  $0\le {\bar u}\in S^\infty (\s,p)$. 
  \item $\|{\bar u} \|_\infty =\|u\|_\infty$. 
  \end{enumerate} 
 \end{theorem} 
 Note that if $p$ is steady, then $u(x)= {\bar u}(x,t)$ for all $t\in\Z_+$  hence the embedding in the theorem gives $u\in H^\infty(\s,p)$. 

To introduce the next result, recall that that the support of a Borel measure $\mu$ on a metric space $(M,d)$, $\mbox{Supp}(\mu)$,  is defined as 
$$ \mbox{Supp}(\mu) = \{y\in M: \mu (U) >0 \mbox{ if } U\mbox{ is open and }y\in U\}.$$ 
By definition, the support is closed and its complement is a $\mu$-null set.

   We say that a transition function $p$ on $\s$ has a locally finite range if for every $x\in\s$, the set $\{y\in\s:p(x,y)>0\}$ is finite.  
 \begin{theorem}
 	\label{th:martin}
 Suppose that $p$ is a transition function on $\s$ with locally finite range satisfying  Assumption \ref{as:required}, and  that 
  $\tau$ is  a stopping time on $\s$-valued sequences satisfying Assumption \ref{as:asymp}. 
  
  Let  $u\in H_+(\widehat{\s},{\widehat {p^\tau}})$ be such that  $\mbox{Supp}(\mu_u^{\widehat{p^\tau}})\subseteq \mbox{Supp}(\mu_{\bf 1}^{\widehat{p^\tau}})$. 

Then there exists a function ${\bar u}$ on $\s\times \Z_+$ such that 
\begin{enumerate}
	\item ${\bar u}(x,0)=u(x),~x \in \s$. 
	\item  ${\bar u}\in S_+(\s,p)$. 
	\end{enumerate}
 \end{theorem}
 The assumption on the support of $u$ is needed to ensure that one can approximate $u$ through bounded harmonic functions. We  need the local finite range assumption  to show that pointwise  limits of the approximating sequence are indeed harmonic sequences, avoiding a strict inequality in Fatou's lemma.  As for the assumption on the support of the measures, it is known that for random walks on regular trees the support of any harmonic function coincides with the minimal Martin boundary (\cite[Section 8]{Sawyer1997}). Nevertheless, the  assumption on the support of a positive harmonic function does not hold in general, even for transition functions with locally finite range  (\cite[Sections 6,7]{Sawyer1997}).

We now prove  two lemmas we will use to prove Theorem \ref{th:newmain}. The lemmas will be followed by the proof of the Theorem~ \ref{th:newmain} and the proof of Theorem~\ref{th:martin}.

  \begin{lemma}
\label{lem:inv_to_tail}
Under the conditions of Theorem \ref{th:newmain}, for  $A \in \partial({\widehat \s}, \widehat {p^\tau})$ there exists $I_A\in {\cal T} (\s^*,p^*)$ such that 
\begin{equation} 
\label{eq:inv_in_tail} \{\lim_{n\to\infty} y^*_n \in  A\} = I_A,~\p_x-a.s. \mbox{ for all }x\in {\widehat \s}-\{*\}.
\end{equation}  
Furthermore, if $A$ and $A'$ are disjoint,  $I_A$ and $I_{A'}$ are disjoint. 
\end{lemma} 
\begin{proof} 
We split the proof of the lemma into two parts. In the first part, we show that \eqref{eq:inv_in_tail} holds for all $A$ which are intersection of $\partial( \widehat\s, \hpt)$ with an open ball (in the Martin topology on $\widehat \s$ relative to the transition function $\hpt$),  centered at a point in $\partial( \widehat\s, \hpt)$. Once this is proved, we show how the lemma extends to all Borel (in the subspace topology)  subsets of $\partial(\widehat \s, \hpt)$. 

We begin with the first part. Let $B_\epsilon(\zeta)$ be a neighborhood of $\zeta$ in  $\partial( \widehat\s, \hpt)$
 with radius $\epsilon$. Since the topology on $\partial(\widehat \s, \hpt)$ is the induced topology from the compact metric space $\widehat\s\cup \partial( \widehat\s, \hpt)$, a basis for the topology on $\partial( \widehat\s, \hpt) $ is the collection of sets of the form

\begin{equation}
\label{eq:special_As} 
A = B_\epsilon(\zeta)\cap \partial( \widehat\s, \hpt)\mbox{ for some } \epsilon>0\mbox{  and }\zeta \in \partial( \widehat\s, \hpt).
\end{equation} 
 
 Fix  $\zeta$ in $\partial( \widehat\s, \hpt)
$ and $\epsilon>0$. Let  $A = B_\epsilon(\zeta)  \cap \partial( \widehat\s, \hpt)$ and let  $B = B_\epsilon(\zeta) \cap \widehat \s$.  Clearly, 
$$\{ \widehat y_\infty  \in A\}=\{x^*_{\tau_n}\in B \mbox{ eventually}\}$$ because by definition, $x^*_{\tau_n} = \widehat y_n$ for $n\in\Z_+$. Denote the event on the right hand side by $B_\infty$. Therefore instead of dealing with the event on the left hand side, we will work with $B_\infty$. Let 
$$
K_n = \left\{\x : \rho_n(\x) \not\in \langle \tau(\x)\rangle\right\}.
$$
 By assumption,  there exists a subsequence $(n_i:i\in\N)$ such that 
$$ \sum_{i} \p_x(K_{n_i})<\infty,~x\in \s.$$ 
Therefore the event $\Gamma=\{K_{n_i} \mbox{ finitely often}\}$ has $\p_x (\Gamma)=1$ for all $x\in \s$. Write $\Gamma_x = \Gamma \cap \{x_0=x\}$. For each $i$, let $\rho_{i,0} = \rho_{n_i}$, and continue inductively, 
$$
\rho_{i,j+1} = \rho_{i,j} + \tau^* \circ S^{\rho_{i,j}}.
$$
 Observe then that $\rho_{i,j}$ are all ${\cal F}_{n_i}^\infty(\s^*)$-measurable. Let 
$$
C_{i} = \bigcap_{j\in\Z_+} \left\{x^*_{\rho_{i,j}} \in B\right\}
$$
and 
$C=\limsup C_i =  \cap_{k=1}^\infty \cup_{i\ge k} C_i.$
 Therefore $\cup_{i\ge k} C_i \in {\cal F}_{n_i}^\infty(\s^*)$, and since this union is decreasing in $k$, it follows that $C$ belongs to ${\cal F}_{n_k}^\infty(\s^*)=S^{-n_k} ({\cal F}^\infty(\s^*))$ for all $k\in\N$, that is $C\in {\cal T} (\s^*)$. Clearly, on $\Gamma_x$, $C$ implies $B_\infty$, and conversely, on $\Gamma_x$, $B_\infty$ implies $C$. Therefore, 
$$ 
B_\infty = C,~\p_x-\mbox{ a.s.,}~x \in\widehat\s-\{*\}.
$$ 
This proves \eqref{eq:inv_in_tail} for the particular choice of $A$, with $I_A = C$. Note that by construction, if $A$ and $A'$ are disjoint, so are the corresponding $C$ and $C'$. 
 
We continue to the second part. Under the subspace topology,  $\partial(\widehat \s,\hpt)$ is a compact metric space. Therefore  it is separable and  every open set is a countable union of such sets from this basis, and every compact subset is a complement of such a countable union.  
In particular, it follows from the first stage that for any compact set $K\in \partial(\widehat \s,\hpt)$, there exists $I_K\in \ol\T(\s^*,{p^*})$, such that 
$$  \{\widehat y_\infty \in K\}= I_K,~\p_x-\mbox{a.s., for all }x\in \widehat\s-\{*\}.$$  

Note that  $\mu_{\bf 1}^{\hpt}$ is a probability measure on a compact metric space (due to extension explained below Theorem \ref{th:doob59}),  it is regular. So for a fixed  Borel set $A \subseteq\partial(\widehat \s,\hpt)$, there exists an increasing sequence $(Q_j:j\ge 1)$ such that for any $j$ the set $Q_j$ is a compact subsets of  $\partial( \widehat \s, \hpt)$ and $\mu_{\bf 1}^{\hpt} (A-Q_j) \to 0$. Now 
$ \{\widehat y_\infty \in \cup Q_j \}=\cup \{\widehat y_\infty \in Q_j\}$, and therefore, there exists $I_{H} \in {\cal T} (\s^*,p^*)$ such that 
$$ \{\widehat y_\infty \in \cup Q_j\} = I_H,\quad \p_x-\mbox{a.s.}\quad x \in \widehat \s - \{*\}.$$ 
We also observe that 
 $$\p_x ( \widehat y_\infty \in A -\cup Q_j) = \int_{A-\cup Q_j} K^{\hpt}(x,\zeta)d \mu_{\bf 1}^{\hpt} (\zeta)=0,$$ 
 for   $x \in \widehat\s$. 
Note that for $x=*$ or $x\in \s - \widehat \s$, then 
$$\p_x (y^*_\infty \in A - \cup Q_j) = \sum_{y\in \widehat\s-\{*\}} {\hpt}(x,y)\p_y(\widehat y_\infty \in A -\cup Q_j) =0.$$ 
By completeness, it follows that the event $\{y^*_\infty \in A - \cup Q_j\}$ is in $\ol{{\cal T}}(\s^*,p^*)$. From this we conclude that 
$$ \{\widehat y_\infty \in A\}= I_H\quad \p_x \mbox{-a.s.}\quad x \in \widehat\s-\{*\}.$$ 
This completes the proof of the second part, and of the lemma. 
\end{proof} 
\begin{lemma}
\label{lem:supremum}
For $x$ in $\s$ and $t$ in $\Z_+$, let $v(x,t) = \sum_{i=1}^N c_i \p_{x,t} (I_{A_i})$,  
where $c_i\ge 0$ and $A_i\in \partial( \widehat \s,\hpt)$ are disjoint and  $I_{A_i}$ is as in Lemma \ref{lem:inv_to_tail}. Then 
$$\|v\|_\infty = \sup_{x\in \widehat \s-\{*\}}\|v(x,0)\|_{\infty}.$$ 
\end{lemma}
\begin{proof}
  Clearly, $\|v \|_\infty \ge \|v(\cdot,0)\|_\infty$. 
  On the other hand,
  $$ v(x,0) = \sum c_i \p_x (y^*_\infty \in A_i),$$ therefore by Theorem \ref{th:limits}, $\sup_{x \in \widehat \s-\{*\}} v(x,0)=\max c_i$,
  while
  $$ v(x,t) = \sum c_i \p_{x,t} (I_{A_i}) \le \max c_i,$$ 
  because $I_{A_i}$ are disjoint, $\p_{x,t}$-a.s.  for all $(x,t)\in \s^*\times \Z_+$.   
\end{proof} 
\begin{proof}[Proof of Theorem \ref{th:newmain}] 
Without loss of generality, we can assume that $\theta$ is such that  $\sum_{x\in \s} \theta (x) u(x)<\infty$. Thus, extend $u$ to $\s^*$ by letting $u(*) = \sum_{x\in \s} \theta (x) u(x)$, and let $\widehat u$ be the restriction of $u$ to $\widehat \s$. By assumption, there exists $0\le f\in L^1(\mu_1^{\widehat {p^\tau}})$ such that 
$$
\widehat u (x) = \int_{\partial_m( \widehat{\s} , \widehat{p^\tau})}  K^{\widehat{p^\tau}}(x,\zeta) f (\zeta) d \mu_1^{\widehat {p^\tau}}(\zeta).
$$ 
There exists  a nondecreasing sequence of nonnegative simple functions   $(f_n:n\in\N)$ on $\partial_m( \widehat{\s} , \widehat{p^\tau})$ such that $f_n \nearrow f$. Letting 
$$\widehat u_n (x) =  \int_{\partial_m( \widehat{\s} , \widehat{p^\tau})} K^{\widehat{p^\tau}}(x,\zeta) f_n (\zeta) d \mu_1^{\widehat{p^\tau}}(\zeta),$$
it follows from Lemma \ref{lem:inv_to_tail} that there exists a combination $\sum c_i {\bf 1}_{I_{A_i}}$, $I_{A_i}\in {\cal T} (\s^*,p^*)$, such that
$$ \widehat u_n (x) = \sum c_i \p_{x,0} (I_{A_i}),\quad x \in \widehat \s-\{*\}.$$

Let $v_n(x,t) = \sum c_i \p_{x,t} (I_{A_i}),~(x,t)\in \s^*\times \Z_+$. Then $v_n \in  S^\infty(\s^*,p^*)$.  
Since by Lemma \ref{lem:supremum},  
$\|v_n\|_\infty\le \sup_{x\in \widehat \s-\{*\}} \|v_n(x,0)\|_\infty\le \|\widehat u\|_\infty$, we can extract a subsequence $(v_{n_k})$ which converges pointwise. Clearly, $v_\infty (x,0)=u(x)$ on $\widehat \s-\{*\}$. However, it also follows from dominated convergence that  $v_\infty \in S^\infty(\s^*,p^*)$. Furthermore, 
$\|v_\infty\|_\infty \le \|\widehat u\|_\infty$.
\end{proof} 

\begin{remark} Here is an outline of an alternative proof to Theorem  \ref{th:newmain}, based entirely on the construction of the Poisson boundary presented in Section \ref{sec:Poisson} and  avoiding the notion of Martin boundary. The proof was suggested  by the referee. \\ 

Let $u\in H^\infty(\s,p^{\tau})$, and let $\phi_u$ be the element in $L^\infty({\cal PB}(\s,p^\tau))$ obtained through \eqref{eq:martin-approach}:
\begin{equation}
\label{eq:y_embed}
\phi_u ({\y}) = \lim_{n\to\infty} u(y_n), \quad \p_{\theta}\mbox{-a.s.},
\end{equation}
Now $\y$ is a deterministic function of $\x$: $\y = \y(\x)$ through $y_n = x_{\tau_n}$, and therefore one can rewrite the lefthand side as a function of $\x$, ${\widetilde{\phi_u}}(\x) = \phi_u(\y(\x))$. 
The  Borell-Cantelli argument in the heart of Lemma \ref{lem:inv_to_tail} gives a sequence $(\rho_{n_i}:i=1,2,\dots)$ of ${\cal F}_{n_i}^\infty(\s)$-measurable random variables with $n_i \nearrow \infty$, and $\rho_{n_i}  \in \langle \tau \rangle$, eventually $\p_{\theta}$-a.s.  Therefore the righthand side of \eqref{eq:y_embed} is  $\lim_{i\to\infty} u(x_{\rho_{n_i}})$,  $\p_{\theta}$-a.s., and is therefore  $\ol{{\cal T}}(\s,p)$-measurable. Thus, we have obtained an embedding 
$$H^{\infty}(\s,p^\tau) \ni u \hookrightarrow \widetilde{\phi_u} \in L^\infty ({\cal TB}(\s,p)).$$
\end{remark} 

We now prove Theorem \ref{th:martin}. 
\begin{proof}[Proof of Theorem \ref{th:martin}]
	By the assumption on  the support of $\mu_u^{\widehat p^\tau}$, there exists a sequence $f_n\in L^\infty (\mu_{\bf 1}^{\widehat {p^\tau}})$ such that $f_n d\mu_1^{\widehat{p^\tau}}$ converges weakly to $\mu^{\widehat {p^{\tau}}}_{u}$ (see Proposition \ref{pr:approx}). Since for each $x\in \widehat \s$, the  mapping $\zeta \to  K^{\widehat{p^\tau}}(x,\zeta)$ is bounded and continuous on the compact metric space $\partial ({\widehat \s},{\widehat {p^\tau}} )$, it follows that each of the functions 
	$$u_n (x) = \int K^{\widehat{p^\tau}}(x,\zeta) f_n (\zeta) d\mu_{\bf 1}^{\widehat {p^\tau}},~x\in{\widehat \s}$$ 
	is in $H^\infty(\widehat{\s},\widehat{p^\tau})$ and  the sequence $(u_n :n\in\N)$  converges pointwise to $u$. By Theorem \ref{th:newmain}, there exists a function $0\le \ol{u}_n \in S^\infty(\s,p)$ such that  $\ol{u}_n(x,0)=u_n (x),~x \in \s$.   
	Since $p \ol{u}_n(x,t+1)=\ol{u}_n (x,t),~x\in \s$,  it follows that $p(x,y) \ol{u}_n(y,t+1)\le \ol{u}_n(x,t)$ whenever $p(x,y)>0$, or 
	$$ \ol{u}_n (y,t+1)\le \inf \{ \frac{\ol{u}_n(x,t)}{p(x,y)}: x~, p(x,y)>0\}.$$
	Iterating, and using the fact that $p$ is irreducible, we have that for every $(y,t)\in \s\times \Z_+$ there exist $x=x(y,t)\in \s$ and a constant $c(y,t)>0$, both not depending on $n$, such that 
	$\ol{u}_n(y,t) \le c(y,t)\ol{u}_n(x(y,t),0)$. Since $\ol{u}_n (x,0)$ converges to $u(x)$ as $n\to\infty$ for every $x\in \s$,  it follows that the sequence of nonnegative numbers $(\ol{u}_n (y,t):n\in\N)$ is bounded. As a result, there exists a subsequence $(n_j:j\in\N)$ such that $\ol{u}_{n_j}(x,t)$ converges to a finite limit at all $(x,t)\in \s\times\Z_+$. Denote this limit function by $\ol{u}$. Clearly, $\ol{u}(x,0)=u(x)$. Since $p$ is locally finite, 
	$$ p \ol{u} (x,t+1)= \sum_{y} p(x,y)\lim_{j\to\infty} \ol{u}_{n_j} (y,t+1) = \lim_{j\to\infty} p \ol{u}_{n_j}(x,t+1)=\lim_{j\to\infty}\ol{u}_{n_j}(x,t)= \ol{ u}(x,t),$$ 
	completing the proof. 	 	  
\end{proof}

\section{Examples}\label{sec:examples}
In this section, we provide some examples.
\subsection{Deterministic Stopping times} 
Suppose that $\tau=c$. Let $\rho = 1$, we have that $\p_x (\rho_n \in \langle \tau \rangle) =1$ whenever $n+1$ is multiple of $c$.

\subsection{First Passage times}\label{sec:hitting}
Let $A$ be a recurrent set for $p$ and $\tau = \inf\{n \ge1: x_n \in A\}$. Setting $\rho = \tau$ satisfies the condition of Assumtption \ref{as:asymp}. This is the generalization of Furstenberg's result for random walks on groups,  when $A$ is a recurrent subgroup  \cite{Furstenberg1973}.
We will now show how this can be used to study equivalence of bounded harmonic functions on product spaces. Suppose that $p$ is an irreducible and  transient  transition function on $X \times Y$, where $X$ and $Y$ are nonempty countable sets. $X$ may be finite. We will  assume that there exists $o\in X$ such that the set  $A=\{o\}\times Y$ is $p$--recurrent. Write $\x=(\x(1),\x(2))$ for the corresponding Markov chain,  where $\x(1)$ is an $X$-valued process, and $\x(2)$ is a $Y$-valued process. Note that in general neither $\x(1)$ nor $\x(2)$ are Markov chains. Let $\tau$ denote the first passage time to $A$: 
$$ \tau = \inf\{t\ge 1: {x}_t(1) = o\}.$$ 
Then $\rho=\tau$ satisfies Assumption \ref{as:asymp}.
Observe next that $p^\tau$ induces a transition function $\gamma$  on $Y$ through the relation $$\gamma(y,y') = p^{\tau}((o,y),(o,y')),~x,y\in Y.$$
 Given $x\in X$ and $y \in Y$, let $v((x,y)) = \e_{(x,y)} [ u({y}_{\tau})]$. Clearly,  $v((o,y)) = \gamma u (y) = u(y)$. Therefore, $p^{\tau}v((o,y))=\gamma u(y) = u(y) = v ((o,y))$. Next, if $x\ne o$, we have 
$$p^\tau v ((x,y)) = \sum_{y'}p^\tau((x,y),(o,y')) \e_{(o,y')} [ u({ y}_{\tau})]=\sum_{y'}p^{\tau}((x,y),(o,y')) u(y')=v((x,y)).$$  
      
One easy example is $X=\Z$,  $Y=\Z^{d-1}$, $d\ge 3$,  and   $p$ being the simple symmetric random walk on $X \times Y = \Z^d$ and $o=0$.  It is known that the Markov chain is steady, and that all bounded harmonic functions are constants, so in particular, $S^\infty(\s,p)$ consists only of constant functions. Furthermore, the first (or any component) is recurrent. Thus  $\gamma$ is a transition function on  $\Z^{d-1}$ which is symmetric, but not nearest neighbor (in fact, it is easy to see that $\sum_{y\in \Z^{d-1}} \gamma(0,y)|y|=\infty$). From Theorem \ref{th:newmain} we therefore obtain that all bounded harmonic functions for $\gamma$ are constants. 

\subsection{Additive functionals} 
Recall that an additive functional for a Markov chain is a real-valued process $I= (I_n:n \ge 0)$, such that $I_{n+k} = I_n + I_k \circ S^n$ and $I_k$ is measurable with respect to the sigma-algebra generated by the first $k+1$ coordinate functions for all $k$ (enough for $k=0,1$). An example for an additive functional is $I_n = \sum_{k\le n} f(x_k)$ where $f:{\s}\to {\mathbb R}$ is any function. \\
Let $(I_n:n \ge 0)$ be an additive functional for $\x$ that satisfies $\displaystyle\lim_{n\to\infty} I_n = \infty$ $\p_x$-a.s. Let $\tau = \inf \{n\ge 1: I_n > I_{n-1}\}$. Setting $\rho = \tau$ clearly satisfies the condition. Note that the above example is a special case, with the additive functional counting the number of visits to $A$. On the other hand, letting $\tau = \inf \{n: I_n \ge a\}$ for some fixed $a$, in general does not satisfy the condition.

\subsection{A generic choice for $\rho$} 
Let $\s$ be a  free semigroup generated by the finite nonempty set ${\cal G}$. That is, the elements of ${\s}$ are finite sequences of elements in ${\cal G}$, the empty sequence included, denoted by $\emptyset$. Given $x \in {\s}$, we write $ x g$ for the sequence in ${\s}$ obtained by concatenating  $g$ to $x$ from the right.  If $y=xg$, we write $g=x^{-1}y$  and refer to $g$ as the increment. Assume that for any $x \in \s$ and $g\in {\cal G}$ the transition function $p$ is invariant under the action of semigroups that is  $p(x,xg) = p_g$, where $g\to p_{g}$ is any probability measure on ${\cal G}$. 
We will also assume that $\tau$ is a stopping time invariant under the action of semigroup in the following sense
\begin{equation} 
\label{eq:invariant_tau}\tau (\omega_0,\omega_1, \dots) =\tau (x\omega_0,x\omega_1, \dots) = \tau(\emptyset, \omega_0^{-1}\omega_1,\omega_0^{-1} \omega_2,\dots)
\end{equation}
for all $x$ in the semigroup $\s$.
That is, $\tau$ is  a function of the consecutive increments rather than the actual path.
Let $\tau$ satisfy the following two additional conditions: 

\begin{itemize}
\item[i)] $\tau$ is bounded by $M \in \Z_+$; 
\item[ii)] $\p_\emptyset (\tau =1)> 0$. 
\end{itemize}
Let $A=\{g\in {\cal G}:\tau(\emptyset,g)=1\}$. 
Observe that from assumption (ii), $\sum_{g\in A} p_g>0$, and therefore $\p_x$-a.s., given  a path $(\omega_0,\omega_1,\dots)$, its associated sequence of increments $(\omega_0^{-1}\omega_1,\omega_1^{-1}\omega_2,\dots)$ contains infinitely many  runs (of consecutive increments) in $A$ longer than any fixed $k$. By (i), any ``time'' interval of length longer than $M$ contains at least one element in $\langle \tau\rangle$ before its last element. If, in addition, all increments corresponding to this time interval are in $A$, it necessarily follows that the last element is in $\langle \tau \rangle$. This simple idea translates to the following definition of $\rho$:
$$\rho(\x) = \inf\left\{n>3M: x_{i-1}^{-1}x_i \in A,~\mbox{ for all } i =n-2M,\dots,n  \right\}.$$
From the definition of the random walk, $\rho$ is finite $\p_x$-a.s. for any $x\in \s$. Also, because $\tau \le M$, between time   $\rho_n - 2M$ and $\rho_n$ there exists at least one element of $\langle \tau \rangle$, call the first such element $\tau_m$. Since all consecutive increments until time $\rho_n$ are in $A$,  all elements of the sequence $(\tau_m, \tau_m+1, \dots, \rho_n)$ are in $\langle \tau \rangle$.  Thus  $\p_x (\rho_n \in \langle \tau \rangle )=1$ and Assumption \ref{as:asymp} is satisfied. 

The second author and Kaimanovich show that the Poisson boundary of random walks on countable groups preserved under any stopping times. Moreover, they show that the results hold for randomized stopping times with finite logarithmic moment. They first lift random walks on a countable group to a random walk on a free finitely or infinitely countable generated) semigroup whose first step of the random walk is distributed on the generators of the free semigroup, and show their results in this setup and then apply them to prove the results for any countable groups, see \cite{Behrang} and \cite{Forghani-Kaimanovich2016} for more details. However, our result in the above example provides a different proof for special stopping times on finitely generated free semigroups.

\subsection{ $\boldsymbol\rho$ is not $\boldsymbol\tau$}
We want to show that also in general choosing $\rho=\tau$ will not satisfy Assumption~\ref{as:asymp}.  Let $\s$ be the free  semigroup generated by ${\cal G}=\{a,b\}$. Then for every $x$ we define  transition function   on ${\s}$ by $$ p(x,xa)=p(x,xb)=\frac12.$$ 
For any sample path $\x=(x_0,x_1,x_2,\cdots)$ and $x$ in $\s$, define the stopping time  
$$\tau(\x)=\begin{cases}
1 & x_0=x,\ x_1=xa\\
2 & x_0=x,\ x_1=xb.
\end{cases}$$

We claim that $\tau$ does not satisfy Assumption~\ref{as:asymp}.  By contradiction,  suppose that  $\rho=\tau$. 
 Then 
\begin{align*} 
	 \p_x (\rho_n \in \langle \tau \rangle )&=\p_x(\rho_n \in  \langle \tau \rangle | n \in \langle \tau \rangle  )\p_x( n \in \langle \tau \rangle)+\p_x(\rho_n \in  \langle \tau \rangle | n \not \in \langle \tau \rangle  )\p_x( n \not\in \langle \tau \rangle)\\
	 &= 1 \times \p_x( n \in \langle \tau \rangle) + \frac 34 (1- \p_x( n \in \langle \tau \rangle) )\\
	 & = \frac 14(3+\p_x( n \in \langle \tau \rangle)),
	\end{align*}
	where the $\frac 34$ factor on the second line is because if $n \not\in \langle \tau \rangle$, then $n+1 \in \langle \tau \rangle$ and so  $\rho_n\in  \langle \tau \rangle$ if and only if either: i) $\rho_n=n+1$, that is $x_n^{-1}x_{n+1}=a$; or ii)  $\rho_n=n+2$ and $n+2\in \langle \tau \rangle$, that is  $x_{n}^{-1}x_{n+1} = b$ and $x_{n+1}^{-1}x_{n+2}=a$. Finally, letting $\alpha_n = \p_x (n \in \langle \tau \rangle)
$, we have $\alpha_1=\frac 12$, $\alpha_2= \frac 34$ and by conditioning on the first increment, for $n\ge 3$, $\alpha_n=\frac 12 \alpha_{n-1}+ \frac 12 \alpha_{n-2}$. As a result, $\lim_{n\to\infty}\alpha_n = \frac 23$. 	It follows that 
$$ \displaystyle\lim_{n\to\infty}\p_x (\rho_n \in \langle \tau \rangle )= \frac{11}{12}<1.$$ 

\subsection{Delayed stopping} 
We construct a Markov chain and a corresponding stopping time with the following property. There exist states such that for every choice of  $\rho$, $\p_x (\rho_n \in \langle \tau \rangle)<1$, for all $n$ large. Nevertheless,  we can define   $\rho$ so that   $\displaystyle\lim_{n\to\infty} \p_x (\rho_n \in \langle \tau\rangle) =1$ for all $x\in \s$. \\
 Let ${\s}={\mathbb N}\times \{0,1\}$. Let $r$ be any irreducible and transient transition function on ${\mathbb N}$. We will also assume that $r$ is lazy, that is $r(x,x) =\frac 12$ for all $x$.  Let $s$ be a transition function on $\{0,1\}$ such that  $0$ is an absorbing state, $s(0,0)=1$, and $s(1,0)=s(0,1)=\frac12$.
 For $n_1,n_2 \in {\mathbb N}$ and $d_1,d_2 \in \{0,1\}$, let 
$$
p((n_1,d_1),(n_2,d_2)) = r(n_1,n_2)s(d_1,d_2).
$$
The space of sample paths of the Markov chain  with transition function  $p$ can be expressed in terms of the component processes. That is, if $\x$ denotes that chain, then  $\x=(\x(1),\x(2))$, with $\x(1)$ and  $\x(2)$ are two independent sample paths with respect to $r$ and $s$, receptively.  Assume that $A$ and $B$ are disjoint recurrent sets for $r$. Then clearly, both $A$ and $B$ are infinite.  Let $\sigma = \inf\{n\ge 1: x_n(1) \in A\}$, and continue inductively $\sigma_1=\sigma,~ \sigma_{n+1} = \sigma\circ {S^{\sigma_n}}$.  Define  
$$
\tau = \begin{cases} 
\sigma & x_0(2) =0\\ \sigma_{\sigma} & x_0(2)=1 
\end{cases}
$$ 
In other words, if $x_0(2)=0$, we stop when $\x(1)$ hits $A$. Otherwise, we wait until $\sigma$, and then stop after $\x(1)$ hits  $A$ an additional $\sigma-1$ more times. Let $m$ be in $B$. Suppose that $\rho:{\s}^{{\mathbb Z}_+}\to {\mathbb Z}_+$. 
We begin by observing that 
$$\p_{(m,1)}(\rho_n \in \langle \tau\rangle) \le \p_{(m,1)}(\sigma \le n) + \p_{(m,1)} (\rho_n \in \langle \tau\rangle,\sigma>n).$$
Now since $x_0(2)=1$, it follows from the definition of $\tau$, that under $\p_{(m,1)}$, we have $\tau=\sigma_{{\sigma}}$. On the event $\sigma>n$, this automatically implies $\tau_1>n$, and in particular, if $\rho_n \in \langle \tau\rangle $, we must have $\rho_n > \sigma+ (\sigma-1) > 2n$. As a result, the event $\{\rho_n \in \langle \tau \rangle\}\cap  \{\sigma>n \} $ is contained in the event $\{\rho \circ S^n > n\}\cap \{\sigma>n \}$. This, with the Markov property imply:  
\begin{align} \nonumber \p_{(m,1)}(\rho_n \in \langle \tau\rangle) &\le  \p_{(m,1)}(\sigma \le n)  + \p_{(m,1)}(\rho\circ S^n>n,\sigma>n) \\
\label{eq:as} 
& =  \p_{(m,1)}(\sigma \le n) +  \e_{(m,1)} [\p_{\x_n}(\rho>n),\sigma>n].
\end{align}
Since we assume that $\rho$ is finite $\p_{(m,0)}$ a.s., there exists $n_0\in\N$ such that  for all $n\ge n_0$,  $\p_{(m,0)}(\rho>n) <\frac 12$. Under $\p_{m,1}$, the event 
$$
C_n=\{x_0(1)=x_1(1)=\dots=x_n(1)\}\cap \{x_n(2) = 0\}
$$ has probability $2^{-n} \times (1-2^{-n})$, and is  contained in  $\{\sigma>n\}$. Thus for $n\ge n_0$, 
\begin{align*} 
 \e_{(m,1)} [\p_{\x_n}(\rho>n), \sigma>n] &\le  \p_{(m,1)} ( C_n)\times  \p_{(m,0)}(\rho>n)+ \p_{(m,1)} (C_n^c,\sigma>n) \\
 & <\frac 12 \p_{(m,1)}(C_n) +  \p_{(m,1)} (C_n^c ,\sigma>n)\\
 & =  \frac 12 \p_{(m,1)}(C_n) + \p_{(m,1)} (\sigma>n) - \p_{(m,1)}(C_n,\sigma>n)\\
 & =\p_{(m,1)} (\sigma>n)  - \frac 12 \p_{(m,1)}(C_n),
 \end{align*}
 where the last equality is due to the fact that under $\p_{(m,1)}$, $C_n \subseteq \{\sigma>n\}$. Plug this inequality into  \eqref{eq:as} to obtain 
 $$  \p_{(m,1)}(\rho_n \in \langle \tau\rangle) \le 1 - \frac 12 2^{ -n} (1-2^{-n}),~n\ge n_0.$$ 
 
Next we show that choosing $\rho=\sigma$ satisfies $\displaystyle\lim_{n\to\infty} \p_{x} (\rho_n \in \langle \tau \rangle) =1$, for all $x$ (note:  we can also choose $\rho=\tau$). Let $N_{i,j}$ count the number of times $\x(1)$ visits $A$ between times $i$ and $j$. That is, 
$$
N_{i,j} = \sum_{i \le n \le j} {\bf 1}_A (x_n(1)).
$$
Observe that from the definition of $\tau$, if $x^2_k =0$, then the times of the $k$-th, $k+1$-th, etc. hits of $A$ by $\x(1)$ will all be in $\langle \tau \rangle$. Letting $T = \inf \{n: \x_n(2) =0\}$, we have 
$$ \p_x (\rho_n \in \langle \tau \rangle) \ge \sum_{k} \p_x ( T = k,N_{k,n}\ge k).$$ 
Since $A$ is recurrent, $ N_{k,n}\underset{n\to\infty}{ \nearrow}  \infty$, and so the result follows from monotone convergence. 

\subsection{Splitting of Poisson Boundary}\label{splitting}
This final example is of a transition function $p$ and a stopping time  where $S^\infty (\s,p)$ consists only of constant functions, yet $H^\infty(\s,p^{\tau})$ contains at least two linearly independent elements. In particular, there does not exist $\rho$ satisfying Assumption \ref{as:asymp}.\\
 Let $p$ be the transition function of the nearest neighbor symmetric random walk on $\s=\Z^d$, $d\ge 3$ (the assumption on the dimension is to ensure transience of $p$). By Ney and Spitzer \cite{Ney-Spitzer}, both the  Poisson and Martin boundary with respect to $p$ are  trivial: all harmonic functions are constants. Write ${ x}=({ x}(1),\dots, { x}(d))$. Define stopping times $T_{+},T_-,T_0$ as follows: 
$$T_+ = \inf\{n \ge 1:{ x}_n(1)={x}_{0}(1)+ 1\} \hspace{1cm} T_- = \inf\{n \ge 1:{ x}_n(1)={x}_{0}(1)- 1\}$$
and 
$$ T_0=\inf\{n\ge 1:{x}_n(1) ={ x}_0(1)\}.$$ 
Finally, let $T_{+,2} = T_{0}\circ S^{T_{+}}$ and $ T_{-,2} = T_{0}\circ S^{T_{-}}$.  In words, $T_\pm$ is the first time ${x}(1)$ is one unit to the right (for $+$) or to the left (for $-$) of its starting point, and $T_{\pm,2}$ is the second time ${x}(1)$ is one unit to the right ($+$) or one unit to the left ($-$) from its starting location. 

Set 
$$
\tau=\begin{cases} T_+ \wedge T_{-,2}\mbox{ if } x_0(1)\ge 0 \\ T_-\wedge T_{+,2}\mbox{ if } x_0(1) < 0. \end{cases}
$$

Observe that by symmetry, $\p(T_+<T_-)=\frac 12$, and $\p(T_+< T_{-,2})= \frac 12+ \frac 12 \times \frac 12 \times \frac 12=\frac 58$. It immediately follows that $x_{\tau}(1)$ is a nearest neighbor Markov chain on $\Z$ with the following transition probabilities: 
$$ 
r(x,y)= \begin{cases} \frac 58 & x \ge 0,~ y= x+1 \mbox{ or } x<0, ~y=x-1 \\ \frac 38 & x\ge 0,~ y=x-1\mbox{ or } x<0,~ y=x+1.\end{cases}
$$ 

In particular, ${ x}^{\tau}(1)$ is transient with positive drift on the positive half line and negative drift on the negative half line. Clearly,  $\{\lim  { x}_{\tau_n}(1)=+\infty\}$ is invariant with respect to $\x^{\tau}$, therefore the function $\p_x (\lim {x}_{\tau_n}(1) = +\infty)$ is a $p^{\tau}$-harmonic function. It is easy to see that this function is not constant. Therefore the conclusion of Theorem \ref{th:newmain} does not hold, which in turn implies that there does not exist a $\rho$ satisfying Assumption \ref{as:asymp}. Note that with this construction, $p^\tau$ is irreducible.   As simpler, yet not irreducible example can be obtained by  defining $\tau$ as the first time $x(1)$ is one unit away from $0$ than where it started at. In this case, $x_{\tau}(1)$ is the Markov chain on $\Z$ which jumps from $0$ to $\pm 1$ with probability $1/2$ each, and jumps from $x$ to $x+\mbox{sgn}(x)$ with probability $1$ for all $x\ne 0$.
\section{Appendix}
\begin{lemma}
	\label{lem:partition}
	Let $\mu$ be a Borel measure on a compact metric space $(M,d)$. Then for  every $\delta$, there exists a number $n\in\N$ and Borel subsets  $N_0,N_1,\dots,N_n$ of $M$ such that 
	\begin{enumerate} 
		\item $\cup_{j=0}^n N_j = M$. 
		\item $N_j \cap N_i =\emptyset$ for $i\ne j$. 
		\item $N_0$ is contained in the complement of $\mbox{Supp}(\mu)$. 
		\item For $j=1,\dots,n$, the diameter of $N_j$ is $< \delta$.  
		\item $N_j,~j=1,\dots,n$ contains a neighborhood of an element in the support of $\mu$. 
	\end{enumerate} 
\end{lemma}
\begin{proof}
	Let $B(x)$ denote the open ball of center $x$ with radius $\delta/2$. 
	The collection of balls $(B(x):x\in \mbox{Supp}(\mu))$ is an open cover of the compact set $\mbox{Supp}(\mu)$. Therefore it possesses a finite subcover indexed by centers $x_1,\dots,x_n$.  
	Let $N_1 =\{y\in B(x_1):d(y,x_1)\le \min_{j\ne 1} d(y,x_j)\}$. Then $x_1\in N_1$, $N_1$ has nonempty interior, and $x_j\not \in N_1$ for all $j\ne 1$. Continue inductively, letting 
	$$N_{i+1}=\{y\in B(x_{i+1}):d(y,x_{i+1})\le \min_{j\ne i+1} d(y,x_j)\}-\cup_{j\le i}N_i.$$
	Thus, $x_{i+1}\in N_{i+1}$, $N_{i+1}$ has nonempty interior, and $x_j\not \in N_{i+1}$ for all $j\le i+1$. Finally, let $N_0=M- \cup_{j=1}^n N_j$. 
\end{proof}  
\begin{prop}
	\label{pr:approx}
	Let $\mu$ and $\nu$ be two probability measures on the Borel sets of  a compact metric space $(M,d)$, satisfying $\mbox{Supp}(\nu)\subseteq \mbox{Supp}(\mu)$. Then there exists a sequence $(f_k:k\in\N)$ of nonnegative simple functions such that $f_n d\mu$ converges weakly to $\nu$. That is, 
	$$ \lim_{n\to\infty} \int g f_n d\mu = \int g d\nu$$ 
	for every continuous $g$ on $M$. 
\end{prop}
\begin{proof} 
	By lemma \ref{lem:partition} for every $k\in\N$ we can find $n=n(k)$ and a partition  $N_0,N_1,\dots,N_n$ of $M$ such that the diameter of $N_1,\dots,N_n$ is $<\frac{1}{k}$, $\nu(N_0)=\mu(N_0)=0$ and $\mu(N_j)>0$ for $j=1,\dots,n$. 
	Let $f_k$ be the simple function equal to $0$ on $N_0$ and to $\nu(N_j)/\mu(N_j)$ on $N_j$ for  $j=1,\dots,n$. Then $\int f_k d \mu =1$. Next fix a continuous function $g$ on $M$. Then by compactness, $g$ is uniformly continuous. Fix $\epsilon>0$, and choose $k$ to such that 
	$|g(x) - g(y)|<\epsilon$ whenever $d(x,y)<\frac{1}{k}$.  Let $N_0,\dots,N_n$ as above, and let $x_1,\dots,x_n$ be arbitrary elements in $N_1,\dots,N_n$, respectively.  Then 
	$$| \int g d\nu - \sum_{i=1}^n g(x_i) \nu(N_i)| <\epsilon$$ 
	and  
	$$ |\int g f_k d\mu - \sum_{i=1}^n g(x_i) \nu(N_i) | <\epsilon,$$ 
	and therefore 
	$$ |\int gf_k d \mu - \int g d\nu|<2\epsilon,$$ 
	completing the proof. 
\end{proof} 
\bibliographystyle{alpha}
\bibliography{Transformed}

\def\cprime{$'$}
\begin{thebibliography}{FMK18}

\bibitem[Bla55]{Blackwell1955}
David Blackwell.
\newblock On transient {M}arkov processes with a countable number of states and
  stationary transition probabilities.
\newblock {\em Ann. Math. Statist.}, 26:654--658, 1955.

\bibitem[Der76]{Derriennic1976}
Yves Derriennic.
\newblock Lois ``z\'ero ou deux'' pour les processus de {M}arkov.
  {A}pplications aux marches al\'eatoires.
\newblock {\em Ann. Inst. H. Poincar\'e Sect. B (N.S.)}, 12(2):111--129, 1976.

\bibitem[Doo59]{Doob59}
J.~L. Doob.
\newblock Discrete potential theory and boundaries.
\newblock {\em J. Math. Mech.}, 8:433--458; erratum 993, 1959.

\bibitem[Doo90]{Doob53}
J.~L. Doob.
\newblock {\em Stochastic processes}.
\newblock Wiley Classics Library. John Wiley \& Sons Inc., New York, 1990.
\newblock Reprint of the 1953 original, A Wiley-Interscience Publication.

\bibitem[Dyn69]{Dynkin69}
E.~B. Dynkin.
\newblock The boundary theory of {M}arkov processes (discrete case).
\newblock {\em Uspehi Mat. Nauk}, 24(2 (146)):3--42, 1969.

\bibitem[Fel56]{Feller1956}
William Feller.
\newblock Boundaries induced by non-negative matrices.
\newblock {\em Trans. Amer. Math. Soc.}, 83:19--54, 1956.

\bibitem[FK16]{Forghani-Kaimanovich2016}
B.~Forghani and V.~A. Kaimanovich.
\newblock Boundary preserving transformations of random walks.
\newblock {\em in preparation}, 2016.

\bibitem[FMK18]{Forghani-Mallahi2016}
B.~Forghani and K.~Mallahi-Karai.
\newblock Positive harmonic functions of transformed random walks.
\newblock {\em Potential. Analysis, https://doi.org/10.1007/s11118-018-9724-4},
  2018.

\bibitem[For15]{Behrang}
B.~Forghani.
\newblock {\em Transformed random walks}.
\newblock PhD thesis, University of Ottawa, Canada, 2015.

\bibitem[For17]{Behrang2016}
Behrang Forghani.
\newblock Asymptotic entropy of transformed random walks.
\newblock {\em Ergodic Theory Dynam. Systems}, 37(5):1480--1491, 2017.

\bibitem[Fur63]{Furstenberg1963}
Harry Furstenberg.
\newblock A {P}oisson formula for semi-simple {L}ie groups.
\newblock {\em Ann. of Math. (2)}, 77:335--386, 1963.

\bibitem[Fur73]{Furstenberg1973}
Harry Furstenberg.
\newblock Boundary theory and stochastic processes on homogeneous spaces.
\newblock In {\em Harmonic analysis on homogeneous spaces ({P}roc. {S}ympos.
  {P}ure {M}ath., {V}ol. {XXVI}, {W}illiams {C}oll., {W}illiamstown, {M}ass.,
  1972)}, pages 193--229. Amer. Math. Soc., Providence, R.I., 1973.

\bibitem[Kai83]{Kaimanovich83}
V.~A. Kaimanovich.
\newblock Differential entropy of the boundary of a random walk on a group.
\newblock {\em Uspekhi Mat. Nauk}, 38(5(233)):187--188, 1983.

\bibitem[Kai91]{Kaimanovich91}
Vadim~A. Kaimanovich.
\newblock Poisson boundaries of random walks on discrete solvable groups.
\newblock In {\em Probability measures on groups, {X} ({O}berwolfach, 1990)},
  pages 205--238. Plenum, New York, 1991.

\bibitem[Kai92]{Ka92}
V.~A. Kaimanovich.
\newblock Measure-theoretic boundaries of {M}arkov chains, {$0$}-{$2$} laws and
  entropy.
\newblock In {\em Harmonic analysis and discrete potential theory ({F}rascati,
  1991)}, pages 145--180. Plenum, New York, 1992.

\bibitem[Kai96]{Kaimanovich1996}
Vadim~A. Kaimanovich.
\newblock Boundaries of invariant {M}arkov operators: the identification
  problem.
\newblock In {\em Ergodic theory of {${\bf Z}^d$} actions ({W}arwick,
  1993--1994)}, volume 228 of {\em London Math. Soc. Lecture Note Ser.}, pages
  127--176. Cambridge Univ. Press, Cambridge, 1996.

\bibitem[Kai00]{Kaimanovich2002}
Vadim~A. Kaimanovich.
\newblock The {P}oisson formula for groups with hyperbolic properties.
\newblock {\em Ann. of Math. (2)}, 152(3):659--692, 2000.

\bibitem[KW02]{Kaimanovich-Woess2002}
Vadim~A. Kaimanovich and Wolfgang Woess.
\newblock Boundary and entropy of space homogeneous {M}arkov chains.
\newblock {\em Ann. Probab.}, 30(1):323--363, 2002.

\bibitem[Muc06]{Muchnik2006}
Roman Muchnik.
\newblock A note on stationarity of spherical measures.
\newblock {\em Israel J. Math.}, 152:271--283, 2006.

\bibitem[NS66]{Ney-Spitzer}
P.~Ney and F.~Spitzer.
\newblock The {M}artin boundary for random walk.
\newblock {\em Trans. Amer. Math. Soc.}, 121:116--132, 1966.

\bibitem[Rev84]{Revuz1984}
D.~Revuz.
\newblock {\em Markov chains}, volume~11 of {\em North-Holland Mathematical
  Library}.
\newblock North-Holland Publishing Co., Amsterdam, second edition, 1984.

\bibitem[Roh52]{Rohlin52}
V.~A. Rohlin.
\newblock On the fundamental ideas of measure theory.
\newblock {\em Amer. Math. Soc. Translation}, 1952(71):55, 1952.

\bibitem[Saw97]{Sawyer1997}
Stanley~A. Sawyer.
\newblock Martin boundaries and random walks.
\newblock In {\em Harmonic functions on trees and buildings ({N}ew {Y}ork,
  1995)}, volume 206 of {\em Contemp. Math.}, pages 17--44. Amer. Math. Soc.,
  Providence, RI, 1997.

\bibitem[Wil90]{Willis1990}
G.~A. Willis.
\newblock Probability measures on groups and some related ideals in group
  algebras.
\newblock {\em J. Funct. Anal.}, 92(1):202--263, 1990.

\bibitem[Woe09]{Woess09}
Wolfgang Woess.
\newblock {\em Denumerable {M}arkov chains}.
\newblock EMS Textbooks in Mathematics. European Mathematical Society (EMS),
  Z\"urich, 2009.
\newblock Generating functions, boundary theory, random walks on trees.

\end{thebibliography}
\end{document}